\documentclass[11pt,a4paper]{article}
\usepackage[latin1]{inputenc}
\usepackage{amsmath}
\usepackage{amsthm}
\usepackage{amsfonts}
\usepackage{amssymb,url,hyperref}
\usepackage{graphicx}
\usepackage[margin=2.7cm]{geometry}
\usepackage{xcolor}
\newtheorem{theorem}{Theorem}
\newtheorem{lemma}{Lemma}

\author{ Mahadi Ddamulira$^{1,*}$, Paul Emong$^{2}$, and Geoffrey Ismail Mirumbe$^{3}$ }
\title{Palindromic concatenations of two distinct repdigits in Narayana's cows sequence$ ^\dagger $}
\date{}

\begin{document}
\maketitle

\begin{abstract}
\noindent Let $(N_{n})_{n\geq 0}$ be Narayana's cows sequence given by a recurrence relation $ N_{n+3}=N_{n+2}+N_n $ for all $ n\ge 0 $, with initial conditions $ N_0=0 $, and $ N_1= N_2=1 $. In this paper, we find all members in Narayana's cow sequence that are palindromic concatenations of two distinct repdigits. Our proofs use techniques on Diophantine approximation which include the theory of linear forms in logarithms of algebraic numbers and Baker's reduction method. We show that $595$ is the only member in Narayana's cow sequence that is a palindromic concatenation of two distinct  repdigits in base $10$. 
\end{abstract}



\noindent
\thanks{$ ^\dagger $This research was supported by Africa-UniNet,  financed by the Austrian Federal Ministry of Education, Science and Research (BMBWF) and implemented by OeAD.}

\noindent 
\thanks{$ ^{*} $ Corresponding author}

\section{Introduction.}

\noindent Narayana's cows sequence $(N_{n})_{n\ge 0}$ was named after the Indian mathematician Narayana \cite{Alou}, who proposed the problem of a herd of cows and calves. He described the sequence as the number of cows present each year, starting from one cow in the first year, where every cow has one calf every year starting in its fourth year.
\medskip

\noindent Narayana's cow sequence is a sequence of positive integers and its first few terms are $$1,1,1,2,3,4,6,9,13,19,28,41,60,88,129,189,277,406,595,872,1278,1873, \ldots.$$ In this sequence, the number of cows in each year is equal to the number of cows in the previous year, plus the number of cows three years ago.
Translating this description into recurrence sequences, we have that in the $n^{\text{th}}$ year, Narayana's cows sequence problem is given by the following recurrence relation:
$$N_{n+3} = N_{n+2} + N_{n},$$
\text{for all $n \geq 0$, with the initial conditions $N_{0} = 0$, and $N_{1} =N_{2} = 1$}.
\medskip

\noindent Let $q\geq 2$ be an integer. We say that a positive integer $R$ is called a base \textit{q-repdigit} if all its digits are equal in base $q$. If $q=10$, then we omit the base and simply say that $R$ is a repdigit. That is, $R$ is of the form
\begin{align*}
R = \overline{\underbrace{b\cdots b}_\text{$k$ times}} = b \left( \frac{10^{k} - 1}{9} \right),
\end{align*}
for some non negative integers $b$ and $k$ with $1 \leq b \leq 9$ and $k \geq 1$. Given $s\geq 1$, we say that $R$ is a concatenation of $s$ repdigits if $R$ can be written in the form $$\overline{\underbrace{b_1\cdots b_1}_\text{$k_{1}$ times} \underbrace{b_2\cdots b_2}_\text{$k_{2}$ times} \cdots \underbrace{b_s\cdots b_s}_\text{$k_{s}$ times} }.$$
\medskip

\noindent Diophantine equations involving repdigits have been studied in several papers by different authors. For instance, the problem of writing repdigits as sum of three Fibonacci numbers has been studied by Luca \cite{FLO}. Alvarado and Luca \cite{Alv} solved the Diophantine equation which involved expressing Fibonacci numbers as sum of two repdigits. In \cite{MD2, Dda}, Ddamulira studied the problem of searching for Padovan and tribonacci numbers which are concatenations of two repdigits, respectively. In \cite{Ala}, the authors proved that the only Fibonacci numbers that are concatenations of two repdigits are $\{13, 21, 34, 55, 89, 144, 233, 377\}$. Erduvan and Keskin, in their paper obtained all Lucas numbers which are concatenations of two repdigits, see \cite{Eduv1}. The same authors determined Lucas numbers which are concatenations of three repdigits \cite{Eduv2}. Other results on concatenation of repdigits appear in  \cite{BH,Brav2,Q,Ray,Tr}. Furthermore, the authors in $\cite{DEM}$ showed that $\{13, 19, 28, 41, 60, 88, 277\}$ are the only members of Narayana's cow sequence that are concatenation of two repdigits.
\medskip

\noindent In this paper, we study a related problem and find all members in Narayana's cow sequence which are palindromic concatenations of two distinct repdigits. To be more specific, we solve the Diophantine equation
\begin{align}\label{eq1}
N_n = \overline{\underbrace{f_1\cdots f_1}_\text{$u_{1}$ times} \underbrace{f_2\cdots f_2}_\text{$u_{2}$ times}\underbrace{f_1\cdots f_1}_\text{$u_{1}$ times} },
\end{align}
 with $f_1 \neq f_2 \in \{0,1, 2, 3,\ldots, 9 \},\ f_{1}>0$, and $ u_{1},\ u_{2}\geq 1$.
\medskip

\noindent Therefore, we have the following result:
\begin{theorem}\label{thm1x}
The only member in Narayana's cows sequence which is a palindromic concatenation of two distinct repdigits is $595$.
\end{theorem}

\section{Preliminary results.}
In this section, we gather the tools needed to prove Theorem $\ref{thm1x}$.

\subsection{Some properties of the Narayana's cows sequence.}
\noindent We consider the characteristic polynomial of the Narayana's cows sequence defined by \\$f(x) := x^3 - x^2 - 1$. Its roots are $\alpha_{1},~\alpha_{2},~\text{and}~\alpha_{3} = \overline{\alpha_{2}}$, where $|\alpha_{2}|=|\alpha_{3}|<1$. We also note that $f(x)$ is irreducible in $\mathbb{Q}[x]$. The Binet formula for $N_{n}$ is given by
\begin{align}\label{eq2}
	N_n = a_{1}\alpha_{1}^{n} + a_{2}\alpha_{2}^{n} + a_{3}\alpha_{3}^{n},
\end{align}
 and holds for all $n \geq 0$, where
\begin{align*}
	a_{1} = \dfrac{\alpha_{1}}{(\alpha_{1}-\alpha_{2})(\alpha_{1}-\alpha_{3})},~a_{2}=\dfrac{\alpha_{2}}{(\alpha_{2}-\alpha_{1})(\alpha_{2}-\alpha_{3})},~~ \text{and}~~ a_{3}= \dfrac{\alpha_{3}}{(\alpha_{3}-\alpha_{1})(\alpha_{3}-\alpha_{2})}=\overline{a_{2}}.
\end{align*}
\medskip
 
 \noindent Numerically, we have the following:
\begin{align*}
	1.46 < & \alpha_{1} < 1.47, \\
	0.82 < |\alpha_{2}| = & |\alpha_{3}| < 0.83,\\
	0.41 < & a_{1} < 0.42,\\
	0.27 < |a_{2}| = & |a_{3}| < 0.28.
\end{align*}
 Considering the estimates for $\alpha_{1},\alpha_{2},\alpha_{3},a_{1},a_{2},a_{3}$, and putting
\begin{align*}
	r(n) := N_n - a_{1}\alpha_{1}^{n} =  a_{2}\alpha_{2}^{n} + a_{3}\alpha_{3}^{n},~ \text{we have that}~ |r(n)| < \frac{1}{\alpha_{1}^{n/2}} \quad \text{for all} \quad n \geq 1.
\end{align*}
Furthermore, the minimal polynomial of $ a_1 $ in $ \mathbb{Z}[x] $ is $ 31x^3-3x-1 $, and  the inequality
\begin{align}\label{eq3}
\alpha_{1}^{n-2} \leq N_n \leq \alpha_{1}^{n-1},
\end{align}
holds for all positive integers $n\geq 1$, and this can be proved by induction.
\subsection{Linear forms in logarithms.}
\noindent Consider the algebraic number $\lambda$ of degree $d$. Let $c> 0$ be the leading coefficient of its minimal polynomial over integers, and $\lambda=\lambda^{(1)},\cdots,\lambda^{(d)}$ denote its conjugates.
 The logarithmic height of $\lambda$ is defined by
\begin{align*}
h(\lambda) := \dfrac{1}{d} \left( \log c + \sum_{i = 1}^d \log \left( \max \{|\lambda^{(i)}|, 1  \}  \right)   \right).
\end{align*}
If $\lambda = \dfrac{p}{q}$ is a rational number with $\gcd(p,q)=1$, and $q > 0$, then $h(\lambda) = \log \max \{ |p|,q \}.$
\medskip

\noindent The logarithmic height has the following properties:
\begin{itemize}
\item [(i)] $h(\lambda_{1} \pm \lambda_{2})  \le h(\lambda_{1}) + h(\lambda_{2}) + \log 2,$
\item [(ii)] $h(\lambda_{1} \lambda_{2} ^{\pm})  \le h(\lambda_{1}) + h(\lambda_{2}),$
\item [(iii)] $h(\lambda^v)  = |v| h(\lambda), ~ \text{for}~ v \in \mathbb{Z}.$
\end{itemize}
\noindent The following is a result of Matveev \cite{MATV}; also see Bugeaud, Mignotte, and Siksek (\cite{BMS}, Theorem 9.4).

\begin{theorem}\label{thm2}
Let $\lambda, \ldots, \lambda_s$ be positive real algebraic numbers in a real algebraic number field $\mathbb{K} \subset \mathbb{R}$ of degree $d$. Let $b_1, \ldots, b_s$ be nonzero integers such that 
\begin{align*}
\varGamma := \lambda_{1} ^{b_1} \ldots \lambda_s ^{b_s} - 1 \neq 0.
\end{align*}
Then
\begin{align*}
\log | \varGamma| > - 1.4 \times 30^{s+3} \times s^{4.5} \times d^{2} (1 + \log d)(1 + \log D)A_1 \ldots A_s,
\end{align*}
where
\begin{align*}
D \geq \max \{|b_1|, \ldots, |b_s| \},
\end{align*}
and
\begin{align*}
A_j \geq \max \{ d h(\lambda_{j}), |\log \lambda_{j}|, 0.16  \}, \quad \text{for all} \quad j = 1, \ldots,s.
\end{align*}
\end{theorem}

\noindent We recall the following result by Guzmán and Luca that appears in \cite{GSL}.
\begin{lemma}\label{l1}
	Let $m \geq 1$ and $T > 0$ be such that $T > (4m^2)^m$ and $T > x/(\log x)^m$. Then
	\begin{align*}
		x < 2^m T (\log T)^m.
	\end{align*}
\end{lemma}
\subsection{Reduction method.}
We need a version of the reduction method of Baker and Davenport \cite{BAK} based their lemma. Here, 
we use the one given by Bravo and Luca \cite{Brav}, which is a slight variation of a result due to  Dujella and Peth\H{o} (see \cite{DP}, Lemma 5a). This is useful in reducing the bounds arising from applying Theorem $\ref{thm2}$.
\begin{lemma}\label{l2}
Let $\kappa \neq 0, A, B$ and $\mu$ be real numbers such that $A > 0$ and $B > 1$. Let $M > 1$ be a positive integer and suppose that $\dfrac{p}{q}$ is a convergent of the continued fraction expansion of $\tau$ with $q > 6M$. Let 
\begin{align*}
\xi := \parallel \mu q \parallel - M \parallel \vartheta q \parallel,
\end{align*}
where $\parallel \cdot \parallel$ denotes the distance from the nearest integer. If $\xi > 0$, then there is no solution of the inequality
\begin{align*}
0 < |m \vartheta - n + \mu| < AB^{-\kappa}
\end{align*}
in positive integers $(m,n,\kappa)$ with
\begin{align*}
\dfrac{\log (Aq/\xi)}{\log B} \leq \kappa \quad \text{and} \quad m \leq M.
\end{align*}
\end{lemma}

\section{Proof of Theorem $\ref{thm1x}$.}
\subsection{The small range.}
A computer search with Mathematica program for solutions of \eqref{eq1} in the range $f_{1}>0, \\ f_{2}\geq 0,\ \text{for} \ f_1 \neq f_2 \in \{0,1, 2,\ldots, 9 \}$, $ 1\le u_{1}, u_2\le 100$, and  $ n\le 500$, revealed only one solution stated in Theorem \ref{thm1x}. From now on, we assume that $n > 500$.
\subsection{Bounding $n$.}
First, we consider the Diophantine equation $(\ref{eq1})$, and rewrite it as
\begin{align}\label{eq4}
N_n & = \overline{\underbrace{f_1\cdots f_1}_\text{$u_{1}$ times} \underbrace{f_2\cdots f_2}_\text{$u_{2}$ times} \underbrace{f_1\cdots f_1}_\text{$u_{1}$ times}}, \nonumber \\
& = \overline{\underbrace{f_1\cdots f_1}_\text{$u_{1}$ times}}\cdot10^{u_{1}+u_{2}} + \overline{ \underbrace{f_2\cdots f_2}_\text{$u_{2}$ times}}\cdot10^{u_{1}}+ \overline{\underbrace{f_1\cdots f_1}_\text{$u_{1}$ times}},\nonumber \\
&=f_1\left(\dfrac{10^{u_{1}}-1}{9}\right)\cdot 10^{u_{1}+u_{2}}+f_2\left(\dfrac{10^{u_{2}}-1}{9}\right)\cdot 10^{u_{1}}+f_1\left(\dfrac{10^{u_{1}}-1}{9}\right),\nonumber \\
& = \dfrac{1}{9} \left(f_1\cdot10^{2u_{1} + u_{2}} - (f_1 - f_2)\cdot 10^{u_{1}+u_{2}} +(f_1-f_2)\cdot10^{u_{1}}- f_1 \right).
\end{align}
Next, we give a result on the size of $n$ versus $2u_{1}+u_{2}$.
\begin{lemma}\label{l4}
All solutions to equation $\eqref{eq4}$ satisfy
\begin{align*}
(2u_{1} + u_{2}) \log 10 - 2 < n \log \alpha_{1} < (u_{1} + u_{2}) \log 10 + 1.
\end{align*}
\end{lemma}
\begin{proof}
\noindent The proof follows from equations $(\ref{eq3})$ and $(\ref{eq4})$, and it can be seen that
\begin{align*}
\alpha_{1}^{n-2} \le N_n < 10^{2u_{1} + u_{2}}.
\end{align*}
Applying logarithms on the above inequality gives 
\begin{align*}
(n-2)\log\alpha_{1}<(2u_{1}+u_{2})\log 10,
\end{align*}
which further yields
\begin{align}\label{eq5}
n\log\alpha_{1}<(2u_{1}+u_{2})\log 10+2\log\alpha_{1}<(2u_{1}+u_{2})\log 10+1.
\end{align}
For the lower bound, one can obtain
\begin{align*}
10^{2u_{1}+u_{2}-1}<N_{n}\leq\alpha_{1}^{n-1},
\end{align*}
and applying logarithms, we have
\begin{align*}
	(2u_{1}+u_{2}-1)\log 10<(n-1)\log\alpha_{1},
\end{align*}
which leads to
\begin{align}\label{eq6}
	(2u_{1}+u_{2})\log 10-2<(2u_{1}+u_{2}-1)\log 10+\log\alpha_{1}<n\log\alpha_{1}.
\end{align}
Comparison of \eqref{eq5} and \eqref{eq6} finishes the proof of Lemma \ref{l4}.
\end{proof}

\noindent We continue to examine \eqref{eq4} in three steps as follows:
\medskip

\noindent \textbf{Step 1.} Equations \eqref{eq2} and \eqref{eq4} give
\begin{align*}
a_{1}\alpha_{1}^{n} + a_{2} \alpha_{2}^{n} + a_{3}\alpha_{3}^{n} =  \dfrac{1}{9} \left(f_1\cdot10^{2u_{1} + u_{2}} - (f_1 - f_2)\cdot 10^{u_{1}+u_{2}} +(f_1-f_2)\cdot10^{u_{1}}- f_1 \right),
\end{align*}
which can be rewritten as
\begin{align*}
9 a_{1}\alpha_{1}^{n}  - f_1\cdot10^{2u_{1} + u_{2}} = -9r(n) - (f_1-f_2)\cdot10^{u_{1}+u_{2}}+(f_1-f_2)\cdot10^{u_{1}}-f_1.
\end{align*}
We can now deduce that
\begin{align*}
|9 a_{1}\alpha_{1}^{n}  - f_1\cdot10^{2u_{1} + u_{2}} | & = |-9r(n) - (f_1-f_2)\cdot10^{u_{1}+u_{2}}+(f_1-f_2)\cdot10^{u_{1}}-f_1| \\
& \leq 9\alpha_{1}^{-n/2}+18\cdot10^{u_{1}+u_{2}} \\
& < 27\cdot10^{u_{1}+u_{2}},
\end{align*}
 and the fact that $n > 500$ has been used. Multiplying both sides of the inequality by \\$f_{1}^{-1}\cdot10^{-2u_{1} - u_{2}}$, we obtain
\begin{align}\label{eqn7}
\left | \left( \dfrac{9a_{1}}{f_1} \right)\cdot\alpha^n\cdot10^{-2u_{1}-u_{2}} - 1  \right | < \dfrac{27\cdot 10^{u_{1}+u_{2}}}{f_1\cdot10^{2u_{1} +u_{2}}} \leq \frac{27}{10^{u_{1}}}.
\end{align}

\noindent We define the expression inside the absolute value on the left-hand side of inequality $(\ref{eqn7})$ by
\begin{align}\label{eq8}
\varLambda_1 := \left(\dfrac{9a_{1}}{f_1} \right)\cdot \alpha_{1}^{n}\cdot10^{-2u_{1} -u_{2}} - 1.
\end{align}
 We compare the upper bound in $(\ref{eq8})$ with the lower bound on the quantity $\varGamma$ given by Theorem $\ref{thm2}$. First, we have that $\varLambda_1 \neq 0$. If $\varLambda_1 = 0$, then it is seen that $$a_{1} \alpha^n = \dfrac{f_{1}}{9}\cdot10^{2u_{1}+u_{2}}.$$ To check this, we consider the $\mathbb{Q}-$automorphism $\sigma$ of the Galois extension $\mathbb{Q}(\alpha_{1},\alpha_{2})$ over $\mathbb{Q}$ defined by $\sigma(\alpha_{1}):=\alpha_{2}$, and $\sigma(\alpha_{2}):=\alpha_{1}$, and have
\begin{align*}
\left|\sigma\left( \dfrac{f_{1}}{9}\cdot10^{2u_{1}+u_{2}} \right)\right| = |\sigma (a_{1} \alpha_{1}^{n})| = |a_{2}\alpha_{2}^{n}| < 1,
\end{align*}
which is false. Using the notations of Theorem $\ref{thm2}$, we have
\begin{align*}
\lambda_1: = \dfrac{9a_{1}}{f_1}, \ \lambda_2: = \alpha_{1}, \ \lambda_3: = 10, \ b_1: = 1, \ b_2: = n, \ b_3: = -2u_{1} - u_{2}, \ s: = 3.
\end{align*} 
Since $2u_{1} + u_{2} < n$, we can take $D := n$. Observe that $\mathbb{Q}(\lambda_1, \lambda_2, \lambda_3) = \mathbb{Q}(\alpha_{1})$, thus, $d:= 3$. Furthermore,
$$
h(\lambda_1) = h \left( \dfrac{9a_{1}}{f_1} \right)
 \le h(9)+h(a_{1})+h(f_{1})
\leq \log 9+\dfrac{1}{3}\log 31+\log 9.
$$
Thus, it follows that $h(\lambda_1) < 5.54$. We also have that $h(\lambda_2) = h(\alpha_{1}) = \displaystyle\dfrac{\log \alpha_{1}}{3}$, and \\$h(\lambda_3) = h(10)= \log 10$. Thus, we can take
\begin{align*}
A_1: = 16.62, \ A_2: = \log \alpha_{1},~\text{and}~ \ A_3: = 3 \log 10.
\end{align*}
 Theorem $\ref{thm2}$ tells us that
\begin{align*}
\log |\Lambda_1| &> -1.4\cdot 30^6 \cdot 3^{4.5}\cdot 3^2 \cdot (1 + \log 3)\cdot (1 + \log n)\cdot 16.62\cdot  (\log \alpha)(3 \log 10)\\& > -1.20\times 10^{14}(1 + \log n), 
\end{align*}
combining with \eqref{eqn7}, we obtain
\begin{align*}
u_{1} \log 10 - \log 27 < 1.20\times 10^{14}(1 + \log n),
\end{align*}
which is simplified to
\begin{align}\label{eq9}
u_{1} \log 10 < 1.21\times 10^{14}(1 + \log n).
\end{align}

\noindent \textbf{Step 2.} Equation $\eqref{eq4}$ is rewritten as 
\begin{align*}
9 a_{1}\alpha_{1}^{n}  - (f_1\cdot10^{u_{1}}-(f_1-f_2))\cdot10^{u_{1}+u_{2}}= -9r(n) + (f_1-f_2)\cdot10^{u_{1}}-f_1,
\end{align*}
which gives us
\begin{align*}
|9 a_{1}\alpha_{1}^{n}  - (f_1\cdot10^{u_{1}}-(f_1-f_2))\cdot10^{u_{1}+u_{2}}| & = | -9r(n) + (f_1-f_2)\cdot10^{u_{1}}-f_1| \\
& \leq 9\cdot\alpha_{1}^{-n/2}+18\cdot10^{u_{1}}<27\cdot10^{u_1}.
\end{align*}
If we divide both sides of the inequality by $(f_1\cdot10^{u_{1}}-(f_1-f_2))\cdot10^{u_{1}+u_{2}}$, then
\begin{align}\label{eq10}
\left| \left( \dfrac{9a_{1}}{f_1\cdot10^{u_{1}}-(f_1-f_2)} \right)\cdot\alpha_{1}^{n}\cdot10^{-u_{1}-u_{2}}- 1 \right| & < \dfrac{27\cdot10^{u_{1}}}{(f_1\cdot10^{u_{1}}-(f_1-f_2))\cdot10^{u_{1}+u_{2}}}< \dfrac{27}{10^{u_2}}.
\end{align}
Set
\begin{align*}
\varLambda_2 :=  \left( \dfrac{9a_{1}}{f_1\cdot10^{u_{1}}-(f_1-f_2)} \right)\cdot\alpha_{1}^{n}\cdot10^{-u_{1}-u_{2}}- 1.
\end{align*}
Similar to the previous case as in $\varLambda_{1}$, we have that $\varLambda_{2}\neq 0$. If $\Lambda_{2}= 0$, then  
\begin{align*}
a_{1}\alpha_{1}^{n} = \left( \dfrac{f_1\cdot10^{u_{1}} - (f_1 - f_2)}{9} \right)\cdot10^{u_{1}+u_{2}}.
\end{align*}
This can be seen when we take the $\mathbb{Q}-$automorphism $\sigma$ of the Galois extension $\mathbb{Q}(\alpha_{1},\alpha_{2})$ over $\mathbb{Q}$ defined by $\sigma(\alpha_{1}):=\alpha_{2}$, and $\sigma(\alpha_{2}):=\alpha_{1}$. Thus, we have
\begin{align*}
\left| \left( \dfrac{f_1\cdot10^{u_{1}} - (f_1 - f_2)}{9} \right)\cdot10^{u_{1}+u_{2}} \right| = |\sigma( a_{1}\alpha_{1}^{n})| = | a_{2}\alpha_{2}^{n} | < 1,
\end{align*}
which is not true. We apply Theorem \ref{thm2} with the obvious notations as follows:
\begin{align*}
\lambda_1 := \dfrac{9a_{1}}{f_1\cdot10^{u_{1}}-(f_1-f_2)}, \ \lambda_2 := \alpha_{1}, \ \lambda_3: = 10, \ b_1: = 1, \ b_2: = n, \ b_3 := -u_{1}-u_{2}, \ s: = 3.
\end{align*}
We also have that $u_{2}<n$, and therefore, we can have $D:=n$. For $\lambda_{1}$, we use the properties of logarithmic heights to deduce that
 
\begin{align*}
h(\lambda_1) & = h \left(\dfrac{9a_{1}}{f_1\cdot10^{u_{1}}-(f_1-f_2)} \right)\\
&\leq h(9)+h(a_{1})+u_{1}h(10)+h(f_1)+h(f_1-f_2)+\log 2\\
&\leq 3\log 9 +h(a_{1})+u_{1} \log 10+\log 2\\
& < 3\log 9+\displaystyle\frac{1}{3}\log 31 + 1.21\times 10^{14}(1 + \log n)\\
& < 1.22\times 10^{14}(1 + \log n).
\end{align*}
We also have that
\begin{align*}
	|\log \lambda_1| & = \bigg|\log \left(\dfrac{9a_{1}}{f_1\cdot10^{u_{1}}-(f_1-f_2)} \right)\bigg |\\
	&\leq \log a + \log 9 + |\log(f_1\cdot10^{u_{1}}-(f_1-f_2))|\\
	&\leq \log a_{1} + \log 9 + \log (f_1\cdot10^{u_{1}}) + \bigg| \log\left(1-\dfrac{(f_1-f_2)}{(f_1\cdot10^{u_{1}})}\right)\bigg |\\
	&\leq u_1\log 10 + \log f_1 + \log 9 + \log (0.55) + \bigg| \dfrac{|(f_1-f_2)|}{(f_1\cdot10^{u_{1}})} + \dfrac{1}{2}\left(\dfrac{|(f_1-f_2)|}{(f_1\cdot10^{u_{1}})}\right)^{2} + \cdots\bigg |\\
	&\leq u_1\log 10 + 3\log 9 + \dfrac{1}{10^{u_{1}}} + \dfrac{1}{2.10^{2u_{1}}} + \cdots\\
	& < 1.21\times 10^{14}(1 + \log n) + 3\log 9 + \dfrac{1}{10^{u_{1}}-1}\\
	& < 1.22\times 10^{14}(1 + \log n).
\end{align*}
Note that \eqref{eq9} has been used in the second last inequality. It is clear that $d\cdot h(\lambda_{1}) > |\log \lambda_{1}|$.
\medskip

\noindent Thus, we have 
\begin{align*}
A_1: = 3.66 \times 10^{14} (1 + \log n),\ A_2: = \log \alpha_{1}, \ \text{and}\ A_3: = 3 \log 10. 
\end{align*}
\medskip

\noindent From Theorem \ref{thm2}, we obtain
\begin{align*}
\log |\varLambda_2| & > -1.4\cdot 30^6\cdot  3^{4.5} \cdot 3^2 \cdot (1 + \log 3)(1 + \log n)(3.66 \times 10^{14} (1 + \log n))(\log\alpha_{1})(3\log 10)\\
& > ~-2.64 \times 10^{27}(1 + \log n)^2,
\end{align*}
and combining with $\eqref{eq10}$ leads to
\begin{align}\label{eq11}
u_{2}\log 10 < 2.64 \times 10^{27}(1 + \log n)^2 + \log 27
< 2.65 \times 10^{27}(1 + \log n)^2.
\end{align}

\noindent \textbf{Step 3.}
Rewriting \eqref{eq4}, we have that
\begin{align*}
	9a_{1}\alpha_{1}^{n}-f_1\cdot10^{2u_{1}+u_{2}}+(f_1-f_2)\cdot10^{u_{1}+u_{2}}-(f_1-f_2)\cdot10^{u_{1}}=-9r(n)-d_1.
\end{align*}
This further yields
\begin{align*}
	\bigg | 9a_{1}\alpha_{1}^{n}-(f_1\cdot10^{u_{1}+u_{2}}+(f_1-f_2)\cdot10^{u_{2}}-(f_1-f_2))\cdot10^{u_{1}} \bigg | & = \bigg| -9r(n)-f_1 \bigg |\\
	& \leq \dfrac{9}{\alpha_{1}^{n/2}}+ 9 < 18.
\end{align*}
As a result, we obtain
\begin{align}\label{eq12}
	\bigg| \left(\dfrac{1}{9a_{1}}\right)(f_1\cdot10^{u_{1}+u_{2}}+(f_1-f_2)\cdot10^{u_{2}}-(f_1-g_2))\cdot\alpha_{1}^{-n}\cdot10^{u_{1}}-1 \bigg|< \dfrac{18}{9a_{1}\alpha_{1}^{n}} < \dfrac{2}{\alpha_{1}^{n}}.
\end{align}
\medskip

\noindent We let
\begin{align*}
\varLambda_{3}:=	\bigg ( \left(\dfrac{1}{9a_{1}}\right)(f_1\cdot10^{u_{1}+u_{2}}+(f_1-f_2)\cdot10^{u_{2}}-(f_1-f_2)) \bigg )\cdot\alpha_{1}^{-n}\cdot10^{u_{1}}-1.
\end{align*}

\noindent Using similar arguments as in $\varLambda_{1}$ and $\varLambda_{2}$, we notice that $\varLambda_{3}\neq 0$. If $\varLambda_{3}= 0$, then we would have
\begin{align*}
	a_{1}\alpha_{1}^{n} = \dfrac{1}{9}\bigg(f_1\cdot10^{u_{1}+u_{2}}+(f_1-f_2)\cdot10^{u_{2}}-(f_1-f_2)\bigg)\cdot10^{u_{1}}.
\end{align*}
We verify this by considering the $\mathbb{Q}-$automorphism $\sigma$ of the Galois extension $\mathbb{Q}(\alpha_{1},\alpha_{2})$ over $\mathbb{Q}$ given by $\sigma(\alpha_{1}):=\alpha_{2}$, and $\sigma(\alpha_{2}):=\alpha_{1}$. Hence, we obtain
\begin{align*}
	\left| \dfrac{1}{9}\bigg(f_1\cdot10^{u_{1}+u_{2}}+(f_1-f_2)\cdot10^{u_{2}}-(d_1-d_2)\bigg)\cdot10^{u_{1}} \right| = |\sigma( a_{1}\alpha_{1}^{n})| = | a_{2}\alpha_{2}^{n} | < 1,
\end{align*}
which is false. Thus, applying Theorem \ref{thm2}, we have

	$\lambda_1 := \left[ \dfrac{1}{9a_{1}}\bigg(f_1\cdot10^{u_{1}+u_{2}}+(f_1-f_2)\cdot10^{u_{2}}-(f_1-f_2)\bigg) \right],\  \lambda_2 := \alpha_{1}, \ \lambda_3: = 10,$\\
	
	$ b_1: = 1, \ b_2: = -n, \ b_3 := u_{1}, \ s: = 3.$\\

Similarly, $D:=n$, and $d:=3$. We also have that
\begin{align*}
	h(\lambda_{1})& \leq h(9)+h(a_{1})+h(f_1)+(u_{1}+u_{2})h(10)+h(f_1-f_2)+u_{2}h(10)+h(f_1-f_2)+3\log 2\\
	& \leq5\log 9 + \dfrac{\log 31}{3} + (u_1+u_2)\log 10 + u_{2}\log 10\\
	& \leq 6\log 9 + (u_1+u_2)\log 10 + u_{2}\log 10.	
\end{align*}
Inequalities $(\ref{eq9})$ and $(\ref{eq11})$ give
\begin{align}\label{eq13}
	(u_{1}+u_{2})\log 10 & < 1.21 \times 10^{14}(1+\log n) + 2.65\times 10^{27}(1+\log n)^{2} \nonumber \\
	&< 2.66\times 10^{27}(1+\log n)^{2}.
\end{align}
Hence, we can deduce that
\begin{align*}
h(\lambda_{1}) < 5.32 \times 10^{27}(1+ \log n)^2.
\end{align*}

\noindent Still further, we have
\begin{align*}
|\log \lambda_{1} | & = \bigg| \log\left( \dfrac{1}{9a_{1}}(f_1\cdot10^{u_{1}+u_{2}}+(f_1-f_2)\cdot10^{u_{2}}-(f_1-f_2))\right)\bigg |\\
& \leq \log 9 + \log a_{1} + |\log(f_1\cdot10^{u_{1}+u_{2}}+(f_1-f_2)\cdot10^{u_{2}}-(f_1-f_2))|\\
& \leq 2\log 9 + \log (f_1\cdot10^{u_{1}+u_{2}})+ \bigg|\log\left(1-\dfrac{(f_1-f_2)(10^{u_{2}}-1)}{f_1\cdot10^{u_{1}+u_{2}}}\right) \bigg|\\
& \leq 3\log 9+ (u_{1}+u_{2})\log 10 + \bigg|\log\left(1-\dfrac{(f_1-f_2)(10^{u_{2}}-1)}{f_1\cdot10^{u_{1}+u_{2}}}\right) \bigg|\\
& \leq 3 \log 9 + (u_{1}+u_{2})\log 10 + \bigg|\dfrac{|(f_1-f_2)(10^{u_{2}}-1)|}{f_1\cdot10^{u_{1}+u_{2}}} + \dfrac{1}{2}\left(\dfrac{|(f_1-f_2)(10^{u_{2}}-1)|}{f_1\cdot10^{u_{1}+u_{2}}}\right)^{2}+\cdots\bigg|\\
& \leq 3\log 9 + (u_{1}+u_{2})\log 10 + \dfrac{1}{10^{u_{1}}}+ \dfrac{1}{2\cdot10^{2u_{1}}}+\cdots\\
& < 3\log 9 +(u_{1}+u_{2})\log 10 + \dfrac{1}{10^{u_{1}}-1}.
\end{align*}
Using inequality \eqref{eq13}, we get $|\log\lambda_{1}|< 2.67\times 10^{27}(1+\log n)^{2}$.
It is also seen that $d\cdot h(\lambda_{1}) > |\log \lambda_{1}|$. Therefore, we have $A_{1}:= 8.01\times 10^{27}(1+\log n)^{2},\ A_{2}:= \log\alpha_{1},\ \text{and } A_{3}:= 3\log 10$. Theorem \ref{thm2} leads to $$\log|\varLambda_{3}| > -5.77\times 10^{40}(1+\log n)^{3},$$ and comparing it with inequality \eqref{eq12}, we obtain $$n\log\alpha_{1} -\log 2 < 5.77\times 10^{40}(1+\log n)^{3}.$$ Thus, we infer that $$n<1.53\times 10^{41}(\log n)^{3},$$ and applying Lemma \ref{l1}, we have $$m: = 3,\ x: = n,\ \text{and}\ T: = 1.53 \times 10^{41}.$$ 
Hence, we obtain the inequality
\begin{align*}
n < 2^3 \cdot 1.53 \times 10^{41}\cdot (\log 1.53 \times 10^{41})^3,
\end{align*}
which simplifies to 
\begin{align*}
n < 1.05 \times 10^{48}.
\end{align*}
From Lemma \ref{l4}, one can conclude that
\begin{align*}
2u_{1} + u_{2} < 1.76 \times 10^{47}.
\end{align*}
The following result is a summary of what has been proved.
\begin{lemma}\label{l5}
All solutions to equation $(\ref{eq4})$ satisfy
\begin{align*}
2u_{1} + u_{2} < 1.76 \times 10^{47} \quad \text{and} \quad n < 1.05 \times 10^{48}.
\end{align*}
\end{lemma}

\subsection{Reducing the bounds.}
\noindent We return to the inequality $(\ref{eqn7})$, and let
\begin{align*}
\Gamma_1: = (2u_{1} + u_{2})\log 10 - n\log \alpha_{1} + \log \left(  \dfrac{d_1}{9a_{1}} \right).
\end{align*}
 Inequality $(\ref{eqn7})$ is now rewritten as
\begin{align*}
\left| e^{- \varGamma_1} - 1 \right| < \dfrac{27}{10^{u_{1}}}.
\end{align*}
We see that $-\varGamma_1\neq 0$, since $e^{-\varGamma_1}-1=\varLambda_{1}\neq 0$.
 Assuming that $u_{1}\ge 2$, then the above inequality becomes $$\left| e^{- \varGamma_1} - 1 \right| < \displaystyle\frac{27}{100}<\displaystyle\frac{1}{2}.$$ 

\noindent It is known that
$	\text{if}~ z<0 ~\text{and}~\mid e^{z}-1\mid<\frac{1}{2},~\text{then}~\mid z\mid<2\mid e^{z}-1\mid.$ Consequently, we get the inequality  $$\left|\varGamma_1 \right| < \dfrac{54}{10^{u_{1}}},$$ which leads to 

\begin{align*}
\left| (2u_{1}+u_{2})\left(\dfrac{\log 10}{\log \alpha_{1}}\right) - n + \left( \dfrac{\log (f_1/9a_{1})}{\log \alpha_{1}} \right)  \right| < \dfrac{54}{10^{u_{1}}\log \alpha_{1}}. 
\end{align*}
 It now follows from Lemma $\ref{l2}$ that:
\begin{align*}
\vartheta: = \dfrac{\log 10}{\log \alpha_{1}}, \quad \mu (f_1): = \dfrac{\log (f_{1}/9a_{1})}{\log \alpha_{1}}, \quad 1\leq f_1\leq 9, \quad A := \dfrac{54}{\log\alpha_{1}}, \quad \text{and}~~ B: = 10.
\end{align*}

\noindent We choose $M: = 10^{48}$ as the upper bound on $2u_{1}+u_{2}$. A quick computation with Mathematica program shows that the convergent of $\vartheta$ is such that $ q_{81} > 6M$, and $\xi > 0.0549802 $. Thus, we have that
\begin{align*}
u_{1} \leq \dfrac{\log (54/\log\alpha_{1})q/\xi)}{\log 10} < 52.
\end{align*}
This tell us that $u_{1}\leq 52$. In fact $u_1 < 2$ is also true since $u_1 < 2 < 52$.
\medskip

\noindent Next, for $f_1\neq f_2\in \{0,1,2,\cdots,9\},\ f_1 > 0$, and $1 \leq u_1 \leq 52$, we go back and use  \eqref{eq10}. Let us define
\begin{align*}
	\varGamma_2: = (u_{1} + u_{2})\log 10 - n\log \alpha_{1} + \log \left(  \dfrac{d_1.10^{u_{1}}-(d_1-d_2)}{9a_{1}} \right).
\end{align*}
Rewriting \eqref{eq10} gives the inequality
\begin{align*}
	\left| e^{- \varGamma_2} - 1 \right| < \dfrac{27}{10^{u_{2}}},
\end{align*}
where $-\varGamma_2\neq 0$, since $e^{-\varGamma_2}-1=\varLambda_{2}\neq 0$.
If we assume that $u_{2}\ge 2$, then we also obtain
\begin{align*}
\left| e^{- \Gamma_2} - 1 \right| < \displaystyle\frac{27}{100}<\displaystyle\frac{1}{2}.
\end{align*} 
Consider the fact that $\text{if}~ z<0 ~\text{and}~\mid e^{z}-1\mid<\dfrac{1}{2},~\text{then}~\mid z\mid<2\mid e^{z}-1\mid.$ As a result, we obtain the inequality $$\left|\varGamma_2 \right| < \dfrac{54}{10^{u_{2}}},$$ which further yields 
\begin{align*}
	\left| (u_{1}+u_{2})\left(\dfrac{\log 10}{\log \alpha_{1}}\right) - n + \left( \dfrac{\log ((f_1\cdot10^{u_{1}}-(f_1-f_2))/9a_{1})}{\log \alpha_{1}} \right)  \right| < \dfrac{54}{10^{u_{2}}\log \alpha_{1}}.
\end{align*}
Applying Lemma \ref{l2} gives us
\begin{align*}
	\vartheta: = \dfrac{\log 10}{\log \alpha_{1}}, \quad \mu (f_1,f_2): = \dfrac{\log ((f_1\cdot10^{u_{1}}-(f_1-f_2))/9a_{1})}{\log \alpha_{1}}, \quad A := \dfrac{54}{\log\alpha_{1}}, ~~\text{and}\quad B: = 10.
\end{align*}

\noindent We take $M: = 10^{48}$ as the upper bound on $u_{1}+u_{2}$, since $u_{1}+u_{2} < 2u_{1}+u_{2}$. A quick computation using Mathematica program reveals the convergent of $\vartheta$ such that $q_{81} > 6M$, and $\xi > 0.0000252223 $. Hence, we write
\begin{align*}
	u_{2} \leq \dfrac{\log (54/\log\alpha_{1})q/\xi)}{\log 10} < 57.
\end{align*}
This implies that $u_{2}\leq 57$. Similarly, the  case $u_2 < 2$ holds since $u_2 < 2 < 57$.
\medskip

\noindent Finally, for $f_1\neq f_2\in \{0,1,2,\cdots,9\},~1 \leq u_1 \leq 52$ and $1\leq u_{2}\leq 57$, we consider \eqref{eq12}, then set
\begin{align*}
\varGamma_3:= u_{1}\log 10-n\log\alpha_{1}+\log\left(\dfrac{f_{1}\cdot10^{u_{1}+u_{2}}+(f_1-f_2)\cdot10^{u_{2}}-(f_{1}-f_{2})}{9a_{1}}\right).
\end{align*}
It is noted that \eqref{eq12} is equivalent to the inequality
\begin{align*}
	\left| e^{ \varGamma_3} - 1 \right| < \dfrac{1}{\alpha_{1}^{n}}.
\end{align*}
Similarly, $\varGamma_3\neq 0$, since $e^{\varGamma_3}-1=\varLambda_{3}\neq 0$.
Using the assumption $n > 500$, we clearly see that $$\left| e^{\varGamma_3} - 1 \right| < \displaystyle\dfrac{1}{2}.$$ 

\noindent Again using the fact that $\text{if}~ z<0 ~\text{and}~\mid e^{z}-1\mid<\dfrac{1}{2},~\text{then}~\mid z\mid<2\mid e^{z}-1\mid.$ Thus, it follows that $$\left|\varGamma_3 \right| < \dfrac{2}{\alpha_1^{n}},$$ which yields the inequality
\begin{align*}
	\left| u_{1}\left(\dfrac{\log 10}{\log \alpha_{1}}\right) - n + \left( \dfrac{\log ((f_{1}\cdot10^{u_{1}+u_{2}}+(f_1-f_2)\cdot10^{u_{2}}-(f_{1}-f_{2}))/9a_{1})}{\log \alpha_{1}} \right)  \right| < \dfrac{4}{\alpha_{1}^{n} \log \alpha_{1}}.
\end{align*}
Similarly, we apply Lemma \ref{l2}, with known parameters;
\begin{align*}
\vartheta &: = \dfrac{\log 10}{\log \alpha_{1}}, \quad \mu (f_1,f_2): = \dfrac{\log ((f_{1}\cdot10^{u_{1}+u_{2}}+(f_1-f_2)\cdot10^{u_{2}}-(f_{1}-f_{2}))/9a_{1})}{\log \alpha_{1}}, \\
	  A &:= \dfrac{4}{\log\alpha_{1}}, \quad \text{and} \quad B: = \alpha_{1}.
\end{align*}

\noindent Again, we take $M:= 10^{48}$ as the upper bound for $u_{1}$, since $u_{1} < 2u_{1}+u_{2}$. With the help of Mathematica program, it is seen that convergent of $\vartheta$ is such that $q_{81} > 6M$, and $ \xi \geq 0.000000902495  $. This implies that

\begin{align*}
n\le \dfrac{\log((4/\log\alpha_{1})q/\xi)}{\log\alpha_{1}}< 332.
\end{align*}
Therefore, $ n\le 332 $, which contradicts the assumption that $ n> 500$. This finishes the proof of  Theorem \ref{thm1x}. \qed

\section*{Acknowledgements.}
The authors thank the anonymous referees for the careful reading of the manuscript, as well as the useful comments and suggestions that have greatly improved on the quality of presentation of this paper.

\section*{Declarations.}
The authors declare no conflict of competing interests.

\section*{Addresses}

$^{1}$ Department of Mathematics, Makerere University, P.O. Box 7062 Kampala, Uganda


\noindent
Email: \url{mahadi.ddamulira@mak.ac.ug} 

\vspace{0.5cm}

\noindent
$^{2}$ Department of Mathematics, Makerere University, P.O. Box 7062 Kampala, Uganda

\noindent 
Email: \url{paul.emong@gmail.com}

\vspace{0.5cm}

\noindent
$^{3}$ Department of Mathematics, Makerere University, P.O. Box 7062 Kampala, Uganda

\noindent
Email:\url{ismail.mirumbe@mak.ac.ug}

\end{document}